\documentclass[]{amsart}
\usepackage{mathrsfs}
\usepackage{color}
\usepackage{amsmath}
\usepackage{amsfonts}
\usepackage{amssymb}
 \newtheorem{Theorem}{Theorem}[section]
 \newtheorem{Corollary}[Theorem]{Corollary}
 \newtheorem{Lemma}[Theorem]{Lemma}

 \newtheorem{Remark}[Theorem]{Remark}

 \numberwithin{equation}{section}

\begin{document}

\title[concavity of minimal $L^{2}$ integrals]
 {General concavity of minimal $L^{2}$ integrals related to multiplier ideal sheaves}

\author{Qi'an Guan}
\address{Address: School of Mathematical Sciences,
Peking University, Beijing, 100871, China.}
\email{guanqian@math.pku.edu.cn}

\thanks{}



\date{\today}

\dedicatory{}

\commby{}

%%% ----------------------------------------------------------------------

\begin{abstract}
In this note, we present a general version of the concavity of the minimal $L^{2}$ integrals related to multiplier ideal sheaves.
\end{abstract}

%%% ----------------------------------------------------------------------
\maketitle
%%% ----------------------------------------------------------------------

\section{Introduction}
The multiplier ideal sheaves related to plurisubharmonic functions plays an important role in complex geometry and algebraic geometry
(see e.g. \cite{tian87,Nadel90,siu96,DEL00,D-K01,demailly-note2000,D-P03,Laz04,siu05,siu09,demailly2010}).
We recall the definition of the multiplier ideal sheaves as follows.

\emph{Let $\varphi$ be a plurisubharmonic function (see \cite{demailly-book,Si,siu74}) on a complex manifold.
It is known that the multiplier ideal sheaf $\mathcal{I}(\varphi)$ was defined as the sheaf of germs of holomorphic functions $f$ such that
$|f|^{2}e^{-\varphi}$ is locally integrable (see \cite{demailly2010}).}

In \cite{demailly-note2000}, Demailly posed the so-called strong openness conjecture on multiplier ideal sheaves (SOC for short) i.e. $\mathcal{I}(\varphi)=\mathcal{I}_{+}(\varphi):=\cup_{\varepsilon>0}\mathcal{I}((1+\varepsilon)\varphi)$.
When $\mathcal{I}(\varphi)=\mathcal{O}$, SOC degenerates to the openness conjecture (OC for short) posed by
Demailly-Koll\'{a}r \cite{D-K01}.

The dimension two case of OC was proved by Favre-Jonsson \cite{FM05j},
and the dimension two case of SOC was proved by Jonsson-Mustata \cite{JM12}.
OC was proved by Berndtsson \cite{berndtsson13}.
SOC was proved by Guan-Zhou \cite{GZopen-c},
see also \cite{Lempert14} and \cite{Hiep14}.

In \cite{berndtsson},
Berndtsson establishes an effectiveness result of OC.
Simulated by Berndtsson's effectiveness result of OC,
continuing the solution of SOC \cite{GZopen-c},
Guan-Zhou \cite{GZopen-effect} establish an effectiveness result of SOC.

Recently, we \cite{guan_sharp} establish a sharp version of the effectiveness result of SOC
by considering a concavity property of the minimal $L^{2}$ integrals related to multiplier ideals.

In the present note,
we obtain a general version of the above concavity property.

\subsection{A general concavity property}\label{sec:main}

Let $M$ be a $n-$dimensional Stein manifold,
and let $K_{M}$ be the canonical (holomorphic) line bundle on $M$.
Let $\psi<-T$ be a plurisubharmonic function on $M$,
and let $\varphi$ be a Lebesgue measurable function on $M$,
such that $\varphi+\psi$ is a plurisubharmonic function on $M$,
where $T\in(-\infty,+\infty)$.

We call a positive smooth function $c$ on $(T,+\infty)$ in class $\mathcal{P}_{T}$ if the following three statements hold

(1) $\int_{T}^{+\infty}c(t)e^{-t}dt<+\infty$;

(2) $c(t)e^{-t}$ is decreasing with respect to $t$;

(3) for any compact subset $K\subseteq M$, $e^{-\varphi}c(-\psi)$ has a positive lower bound on $K$.

Especially, if $\varphi\equiv 0$, then (3) is equivalent to $\liminf_{t\to+\infty}c(t)>0$.

\

{Let $Z_{0}$ be a subset of $\{\psi=-\infty\}$ such that $Z_{0}\cap Supp(\{\mathcal{O}/\mathcal{I(\varphi+\psi)}\})\neq\emptyset$.
Let $U\supseteq Z_{0}$ be an open subset of $M$
and let $f$ be a holomorphic $(n,0)$ form on $U$.
Let $\mathcal{F}\supseteq\mathcal{I}(\varphi+\psi)|_{U}$ be a coherent subsheaf of $\mathcal{O}$ on $U$.}

{Denote
\begin{equation*}
\begin{split}
\inf\{\int_{\{\psi<-t\}}|\tilde{f}|^{2}e^{-\varphi}c(-\psi):\exists {\,}\text{open set}{\,} U' &{\,} \text{s.t.}{\,} Z_{0}\subset U'\subset U, {\,}\&{\,}(\tilde{f}-f)\in H^{0}(\{\psi<-t\}\cap U',
\\ &\mathcal{O}(K_{M})\otimes\mathcal{F}){\,}\&{\,}\tilde{f}\in H^{0}(\{\psi<-t\},\mathcal{O}(K_{M}))\},
\end{split}
\end{equation*}}
by $G(t;c)$ ($G(t)$ for short without misunderstanding), where $c\in\mathcal{P}_{T}$,
and $|f|^{2}:=\sqrt{-1}^{n^{2}}f\wedge\bar{f}$ for any $(n,0)$ form $f$.

If there is no holomorphic holomorphic $(n,0)$ form $\tilde{f}$ on $\{\psi<-t\}$ satisfying $(\tilde{f}-f)\in H^{0}(\{\psi<-t\}\cap U, \mathcal{O}(K_{M})\otimes\mathcal{F})$,
then we set $G(t)=-\infty$.

In the present note, we obtain the following concavity of $G(t)$.

\begin{Theorem}
\label{thm:general_concave}
$G(g^{-1}(r))$ is concave with respect to $r\in(0,\int_{T}^{+\infty}c(t)e^{-t}dt]$,
where $g(t)=\int_{t}^{+\infty}c(t_{1})e^{-t_{1}}dt_{1}$, $t\in[T,+\infty)$.
\end{Theorem}

Especially, when $c(t)\equiv 1$ and $A=0$, Theorem \ref{thm:general_concave} degenerates to the concavity of the minimal $L^{2}$ integrals related to multiplier ideals in \cite{guan_sharp} (Proposition 4.1 in \cite{guan_sharp}).

Theorem \ref{thm:general_concave} implies the following.

\begin{Corollary}
\label{coro:linear_equvi}
For any $c\in \mathcal{P}_{T}$,
the following three statements are equivalent

(1) $G(g^{-1}(r))$ is linear with respect to $r\in(0,\int_{T}^{+\infty}c(t)e^{-t}dt]$,
i.e.,
$$G(t)=\frac{G(T)}{\int_{T}^{+\infty}c(t)e^{-t}dt}\int_{t}^{+\infty}c(t_{1})e^{-t_{1}}dt_{1}$$
holds for any $t\in[T,+\infty)$;

(2) $\frac{G(g^{-1}(r_{0}))}{r_{0}}\leq\frac{G(T)}{\int_{T}^{+\infty}c(t)e^{-t}dt}$ holds
for some $r_{0}\in(0,\int_{T}^{+\infty}c(t)e^{-t}dt)$, i.e., $$\frac{G(t_{0})}{\int_{t_{0}}^{+\infty}c(t_{1})e^{-t_{1}}dt_{1}}\leq\frac{G(T)}{\int_{T}^{+\infty}c(t)e^{-t}dt}$$
holds for some $t_{0}\in(T,+\infty)$;

(3) $\lim_{r\to0+0}\frac{G(g^{-1}(r))}{r}\leq\frac{G(T)}{\int_{T}^{+\infty}c(t)e^{-t}dt}$ holds,
i.e.,
$$\lim_{t\to+\infty}\frac{G(t)}{\int_{t}^{+\infty}c(t_{1})e^{-t_{1}}dt_{1}}\leq\frac{G(T)}{\int_{T}^{+\infty}c(t)e^{-t}dt}$$
holds.
\end{Corollary}

\section{Proof of Theorem \ref{thm:general_concave}}

In this section,
we modify some techniques in \cite{guan_sharp} and prove Theorem \ref{thm:general_concave}.

\subsection{$L^{2}$ methods related to $L^{2}$ extension theorem}

Let $c(t)$ be a positive function in $C^{\infty}((T,+\infty))$ satisfying
$\int_{T}^{\infty}c(t)e^{-t}dt<\infty$ and
\begin{equation}
\label{equ:c_A}
\big(\int_{T}^{t}c(t_{1})e^{-t_{1}}dt_{1}\big)^{2}>c(t)e^{-t}
\int_{T}^{t}(\int_{T}^{t_{2}}c(t_{1})e^{-t_{1}}dt_{1})dt_{2},
\end{equation}
for any $t\in(T,+\infty)$,
where $T\in(-\infty,+\infty)$.
This class of functions was denoted by $\mathcal{C}_{T}$.
Especially, if $c(t)e^{-t}$ is decreasing with respect to $t$ and $\int_{T}^{\infty}c(t)e^{-t}dt<\infty$,
then inequality \ref{equ:c_A} holds  (see \cite{guan-zhou13ap}).

In this section,
we present the following Lemma,
whose various forms already appear in \cite{guan-zhou13p,guan-zhou13ap,guan_sharp} etc.:

\begin{Lemma} \label{lem:GZ_sharp}
Let $B\in(0,+\infty)$ and $t_{0}\geq0$ be arbitrarily given.
Let $M$ be a $n-dimensional$ Stein manifold.
Let $\psi<-T$ be a plurisubharmonic function
on $M$.
Let $\varphi$ be a plurisubharmonic function on $M$.
Let $F$ be a holomorphic $(n,0)$ form on $\{\psi<-t_{0}\}$,
such that
\begin{equation}
\label{equ:20171124a}
\int_{K\cap\{\psi<-t_{0}\}}|F|^{2}<+\infty
\end{equation}
for any compact subset $K$ of $M$,
and
\begin{equation}
\label{equ:20171122a}
\int_{M}\frac{1}{B}\mathbb{I}_{\{-t_{0}-B<\psi<-t_{0}\}}|F|^{2}e^{-\varphi}d\lambda_{n}\leq C<+\infty.
\end{equation}
Then there exists a
holomorphic $(n,0)$ form $\tilde{F}$ on $M$, such that,
\begin{equation}
\label{equ:3.4}
\begin{split}
\int_{M}&|\tilde{F}-(1-b(\psi))F|^{2}e^{-\varphi+v(\psi)}c(-v(\psi))d\lambda_{n}\leq C\int_{T}^{t_{0}+B}c(t)e^{-t}dt
\end{split}
\end{equation}
where
$b(t)=\int_{-\infty}^{t}\frac{1}{B}\mathbb{I}_{\{-t_{0}-B< s<-t_{0}\}}ds$,
$v(t)=\int_{0}^{t}b(s)ds$,
and $c(t)\in \mathcal{C}_{T}$.
\end{Lemma}

It is clear that $\mathbb{I}_{(-t_{0},+\infty)}\leq b(t)\leq\mathbb{I}_{(-t_{0}-B,+\infty)}$ and $\max\{t,-t_{0}-B\}\leq v(t) \leq\max\{t,-t_{0}\}$.

\subsection{Some properties of $G(t)$}\label{sec:minimal}
\

Following the notations and assumptions in Section \ref{sec:main}, we present some properties related to $G(t)$
in the present section.

The following lemma is a characterization of $G(T)\neq0$.

\begin{Lemma}
\label{lem:0}
$f\not\in \mathcal{F}(U)\Leftrightarrow G(T)\neq0$ (maybe $+\infty$).
\end{Lemma}

\begin{proof}
It is clear that $f\in \mathcal{F}(U)\Rightarrow G(T)=0$.

In the following part, we prove that $f\not\in \mathcal{F}(U)\Rightarrow G(T)\neq 0$ (maybe $-\infty$ or $+\infty$).
We prove it by contradiction:
if not, then there exists holomorphic $(n,0)$ forms $\{\tilde{f}_{j}\}_{j\in\mathbb{N}^{+}}$ on $M$
such that $\lim_{j\to+\infty}\int_{M}|\tilde{f}_{j}|^{2}e^{-\varphi}c(-\psi)=0$ and
$(f_{j}|_{U}-f)\in H^{0}(U,\mathcal{O}(K_{M})\otimes \mathcal{F})$ for any $j$.
As $e^{-\varphi}c(-\psi)$ has positive lower bound on any compact subset of $M$, there exists a subsequence of $\{\tilde{f}_{j}\}_{j\in\mathbb{N}^{+}}$
denoted by $\{\tilde{f}_{j_{k}}\}_{k\in\mathbb{N}^{+}}$ compactly convergent to $0$.
It is clear that $\tilde{f}_{j_{k}}-f$ is compactly convergent to $0-f=f$ on $U$.
It follows from the closedness of the sections of coherent analytic sheaves under the topology of compact convergence (see \cite{G-R})
that $f\in H^{0}(U,\mathcal{O}(K_{M})\otimes \mathcal{F})$,
which contradicts $f\not\in H^{0}(U,\mathcal{O}(K_{M})\otimes \mathcal{F})$.
Then we obtain $f\not\in H^{0}(U,\mathcal{O}(K_{M})\otimes \mathcal{F})\Rightarrow G(T)>0$ (maybe $+\infty$).
This proves Lemma \ref{lem:0}.
\end{proof}

The following lemma shows the uniqueness of the holomorphic (n,0) form related to $G(t)$.

\begin{Lemma}
\label{lem:A}
Assume that $G(t)<+\infty$ for some $t\in[T,+\infty)$.
Then there exists a unique holomorphic $(n,0)$ form $F_{t}$ on
$\{\psi<-t\}$ satisfying $(F_{t}-f)\in H^{0}(\{\psi<-t\}\cap U,\mathcal{O}(K_{M})\otimes \mathcal{F})$ and $\int_{\{\psi<-t\}}|F_{t}|^{2}e^{-\varphi}c(-\psi)=G(t)$.
Furthermore,
for any holomorphic $(n,0)$ form $\hat{F}$ on $\{\psi<-t\}$ satisfying $(\hat{F}-f)\in H^{0}(\{\psi<-t\}\cap U,\mathcal{O}(K_{M})\otimes \mathcal{F})$ and
$\int_{\{\psi<-t\}}|\hat{F}|^{2}e^{-\varphi}c(-\psi)<+\infty$,
we have the following equality
\begin{equation}
\label{equ:20170913e}
\begin{split}
&\int_{\{\psi<-t\}}|F_{t}|^{2}e^{-\varphi}c(-\psi)+\int_{\{\psi<-t\}}|\hat{F}-F_{t}|^{2}e^{-\varphi}c(-\psi)
\\=&
\int_{\{\psi<-t\}}|\hat{F}|^{2}e^{-\varphi}c(-\psi).
\end{split}
\end{equation}
\end{Lemma}

\begin{proof}
Firstly, we prove the existence of $F_{t}$.
As $G(t)<+\infty$
then there exists holomorphic (n,0) forms $\{f_{j}\}_{j\in\mathbb{N}^{+}}$ on $\{\psi<-t\}$ such that
$\int_{\{\psi<-t\}}|f_{j}|^{2}e^{-\varphi}c(-\psi)\to G(t)$,
and $(f_{j}-f)\in H^{0}(\{\psi<-t\}\cap U,\mathcal{O}(K_{M})\otimes \mathcal{F})$.
Then there exists a subsequence of $\{f_{j}\}$ compact convergence to a holomorphic $(n,0)$ form $F$ on $\{\psi<-t\}$
satisfying $\int_{K}|f|^{2}e^{-\varphi}c(-\psi)\leq G(t)$ for any compact set $K\subset\{\psi<-t\}$,
which implies $ \int_{\{\psi<-t\}}|f|^{2}e^{-\varphi}c(-\psi)\leq G(t)$ by Levi's Theorem.
As $e^{-\varphi}c(-\psi)$ has positive lower bound on any compact subset of $M$,
it follows from
the closedness of the sections of coherent analytic sheaves under the topology of compact convergence (see \cite{G-R})
that $(F-f)\in H^{0}(\{\psi<-t\}\cap U,\mathcal{O}(K_{M})\otimes \mathcal{F})$.
Then we obtain the existence of $F_{t}(=F)$.

Secondly, we prove the uniqueness of $F_{t}$ by contradiction:
if not, there exist two different holomorphic (n,0) forms $f_{1}$ and $f_{2}$ on on $\{\psi<-t\}$
satisfying $\int_{\{\psi<-t\}}|f_{1}|^{2}e^{-\varphi}c(-\psi)=\int_{\{\psi<-t\}}|f_{2}|^{2}=G(t)$,
$(f_{1}-f)\in H^{0}(\{\psi<-t\}\cap U,\mathcal{O}(K_{M})\otimes \mathcal{F})$ and $(f_{2}-f)\in H^{0}(\{\psi<-t\}\cap U,\mathcal{O}(K_{M})\otimes \mathcal{F})$.
Note that
\begin{equation}
\begin{split}
&\int_{\{\psi<-t\}}|\frac{f_{1}+f_{2}}{2}|^{2}e^{-\varphi}c(-\psi)+\int_{\{\psi<-t\}}|\frac{f_{1}-f_{2}}{2}|^{2}e^{-\varphi}c(-\psi)
\\=&
\frac{\int_{\{\psi<-t\}}|f_{1}|^{2}e^{-\varphi}c(-\psi)+\int_{\{\psi<-t\}}|f_{2}|^{2}e^{-\varphi}c(-\psi)}{2}=G(t),
\end{split}
\end{equation}
then we obtain that
$$\int_{\{\psi<-t\}}|\frac{f_{1}+f_{2}}{2}|^{2}e^{-\varphi}c(-\psi)<G(t),$$
and $(\frac{f_{1}+f_{2}}{2}-f)\in H^{0}(\{\psi<-t\}\cap U,\mathcal{O}(K_{M})\otimes \mathcal{F})$, which contradicts the definition of $G(t)$.

Finally, we prove equality \ref{equ:20170913e}.
For any holomorphic $h$ on $\{\psi<-t\}$ satisfying $\int_{\{\psi<-t\}}|h|^{2}e^{-\varphi}c(-\psi)<+\infty$
and $h\in H^{0}(\{\psi<-t\}\cap U,\mathcal{O}(K_{M})\otimes \mathcal{F})$,
it is clear that
for any complex number $\alpha$,
$F_{t}+\alpha h$ satisfying $((F_{t}+\alpha h)-f)\in H^{0}(\{\psi<-t\}\cap U,\mathcal{O}(K_{M})\otimes \mathcal{F})$,
and $\int_{\{\psi<-t\}}|F_{t}|^{2}e^{-\varphi}c(-\psi)\leq\int_{\{\psi<-t\}}|F_{t}+\alpha h|^{2}e^{-\varphi}c(-\psi)<+\infty$.
Note that
$$\int_{\{\psi<-t\}}|F_{t}+\alpha h|^{2}e^{-\varphi}c(-\psi)-\int_{\{\psi<-t\}}|F_{t}|^{2}e^{-\varphi}c(-\psi)\geq 0$$
implies
$$\Re\int_{\{\psi<-t\}}F_{t}\bar{h}e^{-\varphi}c(-\psi)=0$$ by considering $\alpha\to0$,
then
$$\int_{\{\psi<-t\}}|F_{t}+h|^{2}e^{-\varphi}c(-\psi)=\int_{\{\psi<-t\}}(|F_{t}|^{2}+|h|^{2})e^{-\varphi}c(-\psi).$$
Choosing $h=\hat{F}-F_{t}$, we obtain equality \ref{equ:20170913e}.

\end{proof}

The following function shows the lower semi-continuity property of $G(t)$.

\begin{Lemma}
\label{lem:B}
Assume that $G(T)<+\infty$.
Then $G(t)$ is decreasing with respect to $t\in[T,+\infty)$,
such that
$\lim_{t\to t_{0}+0}G(t)=G(t_{0})$ $(t_{0}\in[T,+\infty))$,
$\lim_{t\to t_{0}-0}G(t)\geq G(t_{0})$ $(t_{0}\in(T,+\infty))$,
and $\lim_{t\to +\infty}G(t)=0$, where $t_{0}\in[T,+\infty)$.
Especially $G(t)$ is lower semi-continuous on $[T,+\infty)$.
\end{Lemma}

\begin{proof}
By the definition of $G(t)$,
it is clear that $G(t)$ is decreasing on $[T,+\infty)$ and $\lim_{t\to t_{0}-0}G(t)\geq G(t_{0})$.
It suffices to prove $\lim_{t\to t_{0}+0}G(t)=G(t_{0}).$
We prove it by contradiction:
if not,
then $\lim_{t\to t_{0}+0}G(t)<G(t_{0})$.

By Lemma \ref{lem:A}, there exists a unique holomorphic (n,0) form $F_{t}$ on
$\{\psi<-t\}$ satisfying $(F_{t}-f)\in H^{0}(\{\psi<-t\}\cap U,\mathcal{O}(K_{M})\otimes \mathcal{F})$ and
$\int_{\{\psi<-t\}}|F_{t}|^{2}e^{-\varphi}c(-\psi)=G(t)$.
Note that $G(t)$ is decreasing implies that
$\int_{\{\psi<-t\}}|F_{t}|^{2}e^{-\varphi}c(-\psi)\leq\lim_{t\to t_{0}+0}G(t)$ for any $t>t_{0}$.
As $e^{-\varphi}c(-\psi)$ has positive lower bound on any compact subset of $M$,
for any compact subset $K$ of $\{\psi<-t_{0}\}$,
there exists $\{F_{t_{j}}\}$ $(t_{j}\to t_{0}+0,$ as $j\to+\infty)$
uniformly convergent on $K$.
Then there exists a subsequence of $\{F_{t_{j}}\}$ (also denoted by $\{F_{t_{j}}\}$) convergent on
any compact subset of $\{\psi<-t_{0}\}$.

Let $\hat{F}_{t_{0}}:=\lim_{j\to+\infty}F_{t_{j}}$, which is a holomorphic $(n,0)$ form on $\{\psi<-t_{0}\}$.
Then it follows from the decreasing property of $G(t)$ that
$$\int_{K}|\hat{F}_{t_{0}}|^{2}e^{-\varphi}c(-\psi)\leq \lim_{j\to+\infty}\int_{K}|F_{t_{j}}|^{2}e^{-\varphi}c(-\psi)\leq
\lim_{j\to+\infty}G(t_{j})\leq\lim_{t\to t_{0}+0}G(t)$$
for any
compact set $K\subset \{\psi<-t_{0}\}$.
It follows from Levi's theorem that
$$\int_{M}|\hat{F}_{t_{0}}|^{2}e^{-\varphi}c(-\psi)\leq \lim_{t\to t_{0}+0}G(t).$$
Then we obtain that $G_{t_0}\leq\int_{M}|\hat{F}_{t_{0}}|^{2}e^{-\varphi}c(-\psi)\leq \lim_{t\to t_{0}+0}G(t)$,
which contradicts $\lim_{t\to t_{0}+0}G(t)<G(t_{0})$.
\end{proof}

We consider the derivatives of $G(t)$ in the following lemma.

\begin{Lemma}
\label{lem:C}
Assume that $G(T)<+\infty$.
Then for any $t_{0}\in(T,+\infty)$,
we have
$$\frac{G(T)-G(t_{0})}{\int_{T}^{+\infty}c(t)e^{-t}dt-\int_{t_{0}}^{+\infty}c(t)e^{t}dt}\leq
\frac{\liminf_{B\to0+0}(\frac{G(t_{0})-G(t_{0}+B)}{B})}{c(t_{0})e^{-t_{0}}}.$$
\end{Lemma}

\begin{proof}
By Lemma \ref{lem:A},
there exists a
holomorphic $(n,0)$ form $F_{t_{0}}$ on $\{\psi<t_{0}\}$, such that
$(F_{t_{0}}-f)\in H^{0}(\{\psi<-t_{0}\}\cap U,\mathcal{O}(K_{M})\otimes \mathcal{F})$ and
$\int_{\{\psi<-t_{0}\}}|F_{t_{0}}|^{2}e^{-\varphi}c(-\psi)=G(t_{0})$.

It suffices to consider that $\liminf_{B\to0+0}\frac{G(t_{0})-G(t_{0}+B)}{B}\in(-\infty,0]$
because of the decreasing property of $G(t)$.
Then there exists $B_{j}\to 0+0$ $(j\to+\infty)$ such that
$$\lim_{j\to+\infty}\frac{G(t_{0})-G(t_{0}+B_{j})}{B_{j}}=\liminf_{B\to0+0}\frac{G(t_{0})-G(t_{0}+B)}{B}$$
and $\{\frac{G(t_{0})-G(t_{0}+B_{j})}{B_{j}}\}_{j\in\mathbb{N}^{+}}$ is bounded.

As $t\leq v(t)$, the decreasing property of $c(t)e^{-t}$ shows that
$$c(t)e^{-t}\geq c(-v(-t))e^{v(-t)}$$ for any $t\geq 0$,
which implies
$$e^{-\psi+v(\psi)}c(-v(\psi))\geq c(-\psi).$$
Lemma \ref{lem:GZ_sharp} $(\varphi\sim\varphi+\psi$ here $\sim$ means the former replaced by the latter
and the notation will be used through out the paper$)$
shows that for any $B_{j}$,
there exists
holomorphic $(n,0)$ form $\tilde{F}_{j}$ on $M$, such that
$(\tilde{F}_{j}-F_{t_{0}})\in H^{0}(\{\psi<-t_{0}\}\cap U,\mathcal{O}(K_{M})\otimes \mathcal{I}(\varphi+\psi))
\subseteq H^{0}(\{\psi<-t_{0}\}\cap U,\mathcal{O}(K_{M})\otimes \mathcal{F})$
$(\Rightarrow (\tilde{F}_{j}-f)\in H^{0}(\{\psi<-t_{0}\}\cap U,\mathcal{O}(K_{M})\otimes \mathcal{F})$
and
\begin{equation}
\label{equ:GZc}
\begin{split}
&\int_{M}|\tilde{F}_{j}-(1-b_{t_{0},B_{j}}(\psi))F_{t_{0}}|^{2}e^{-\varphi}e^{-\psi+v(\psi)}c(-v(\psi))
\\\leq&\int_{M}|\tilde{F}_{j}-(1-b_{t_{0},B_{j}}(\psi))F_{t_{0}}|^{2}e^{-\varphi}c(-\psi)
\\\leq&
\int_{T}^{t_{0}+B_{j}}c(t)e^{-t}dt
\int_{M}\frac{1}{B_{j}}(\mathbb{I}_{\{-t_{0}-B_{j}<\psi<-t_{0}\}})|F_{t_{0}}|^{2}e^{-\varphi-\psi}
\\\leq&
\frac{e^{t_{0}+B_{j}}\int_{T}^{t_{0}+B_{j}}c(t)e^{-t}dt}{\inf_{t\in(t_{0},t_{0}+B_{j})}c(t)}
\int_{M}\frac{1}{B_{j}}(\mathbb{I}_{\{-t_{0}-B_{j}<\psi<-t_{0}\}})|F_{t_{0}}|^{2}e^{-\varphi}c(-\psi)
\\\leq&
\frac{e^{t_{0}+B_{j}}\int_{T}^{t_{0}+B_{j}}c(t)e^{-t}dt}{\inf_{t\in(t_{0},t_{0}+B_{j})}c(t)}
\times(\int_{M}\frac{1}{B_{j}}\mathbb{I}_{\{\psi<-t_{0}\}}|F_{t_{0}}|^{2}e^{-\varphi}c(-\psi)
\\&-\int_{M}\frac{1}{B_{j}}\mathbb{I}_{\{\psi<-t_{0}-B_{j}\}}|F_{t_{0}}|^{2}e^{-\varphi}c(-\psi))
\\\leq&
\frac{e^{t_{0}+B_{j}}\int_{T}^{t_{0}+B_{j}}c(t)e^{-t}dt}{\inf_{t\in(t_{0},t_{0}+B_{j})}c(t)}
\times\frac{G(t_{0})-G(t_{0}+B_{j})}{B_{j}}
\end{split}
\end{equation}

Firstly, we will prove that $\int_{M}|\tilde{F}_{j}|^{2}e^{-\varphi}c(-\psi)$ is bounded with respect to $j$.

Note that
\begin{equation}
\label{equ:GZd}
\begin{split}
&(\int_{M}|\tilde{F}_{j}-(1-b_{t_{0},B_{j}}(\psi))F_{t_{0}}|^{2}e^{-\varphi}c(-\psi))^{1/2}
\\\geq&(\int_{M}|\tilde{F}_{j}|^{2}e^{-\varphi}c(-\psi))^{1/2}-(\int_{M}|(1-b_{t_{0},B_{j}}(\psi))F_{t_{0}}|^{2}e^{-\varphi}c(-\psi))^{1/2}
\end{split}
\end{equation}
then it follows from inequality \ref{equ:GZc} that
\begin{equation}
\label{equ:GZe}
\begin{split}
&(\int_{M}|\tilde{F}_{j}|^{2}e^{-\varphi}c(-\psi))^{1/2}
\\\leq&(\frac{e^{t_{0}+B_{j}}\int_{T}^{t_{0}+B_{j}}c(t)e^{-t}dt}{\inf_{t\in(t_{0},t_{0}+B_{j})}c(t)})^{1/2}
(\frac{G(t_{0})-G(t_{0}+B_{j})}{B_{j}})^{1/2}
\\&+(\int_{M}|(1-b_{t_{0},B_{j}}(\psi))F_{t_{0}}|^{2}e^{-\varphi}c(-\psi))^{1/2}.
\end{split}
\end{equation}
Since $\{\frac{G(t_{0}+B_{j})-G(t_{0})}{B_{j}}\}_{j\in\mathbb{N}^{+}}$ is bounded and $0\leq b_{t_{0},B_{j}}(\psi)\leq 1$,
then $\int_{M}|\tilde{F}_{j}|^{2}$ is bounded with respect to $j$.

\

Secondly, we will prove the main result.

It follows from $b_{t_0}(\psi)=1$ on $\{\psi\geq -t_{0}\}$ that
\begin{equation}
\label{equ:20170915b}
\begin{split}
&\int_{M}|\tilde{F}_{j}-(1-b_{t_{0},B_{j}}(\psi))F_{t_{0}}|^{2}e^{-\varphi}c(-\psi)
\\=&\int_{\{\psi\geq-t_{0}\}}|\tilde{F}_{j}|^{2}e^{-\varphi}c(-\psi)+\int_{\{\psi<-t_{0}\}}|\tilde{F}_{j}-(1-b_{t_{0},B_{j}}(\psi))F_{t_{0}}|^{2}e^{-\varphi}c(-\psi)
\end{split}
\end{equation}

Denote that $||\cdot||_{2}:=(\int_{\{\psi<-t_{0}\}}|\cdot|^{2}e^{-\varphi}c(-\psi))^{1/2}$.
It is clear that
\begin{equation}
\label{equ:20170915c}
\begin{split}
&||\tilde{F}_{j}-(1-b_{t_0}(\psi))F_{t_{0}}||_{2}^{2}
\\\geq&(||\tilde{F}_{j}-F_{t_{0}}||_{2}-||b_{t_0,B_{j}}(\psi)F_{t_{0}}||_{2})^{2}
\\\geq&||\tilde{F}_{j}-F_{t_{0}}||_{2}^{2}-2||\tilde{F}_{j}-F_{t_{0}}||_{2}||b_{t_0,B_{j}}(\psi)F_{t_{0}}||_{2}
\\\geq&
||\tilde{F}_{j}-F_{t_{0}}||_{2}^{2}-2||\tilde{F}_{j}-F_{t_{0}}||_{2}
(\int_{\{-t_{0}-B_{j}<\psi<-t_{0}\}}|F_{t_{0}}|^{2}e^{-\varphi}c(-\psi))^{1/2},
\end{split}
\end{equation}
where the last inequality follows from $0\leq b_{t_{0},B_{j}}(\psi)\leq 1$ and $b_{t_{0},B_{j}}(\psi)=0$ on $\{\psi\leq -t_{0}-B_{0}\}$.

Combining equality \ref{equ:20170915b}, inequality \ref{equ:20170915c} and equality \ref{equ:20170913e},
we obtain that
\begin{equation}
\label{equ:20170915e}
\begin{split}
&\int_{M}|\tilde{F}_{j}-(1-b_{t_{0},B_{j}}(\psi))F_{t_{0}}|^{2}e^{-\varphi}c(-\psi)
\\=&\int_{\{\psi\geq-t_{0}\}}|\tilde{F}_{j}|^{2}e^{-\varphi}c(-\psi)+||\tilde{F}_{j}-(1-b_{t_0}(\psi))F_{t_{0}}||_{2}^{2}
\\\geq&
\int_{\{\psi\geq-t_{0}\}}|\tilde{F}_{j}|^{2}e^{-\varphi}c(-\psi)+||\tilde{F}_{j}-F_{t_{0}}||_{2}^{2}-2||\tilde{F}_{j}-F_{t_{0}}||_{2}||b_{t_0,B_{j}}(\psi)F_{t_{0}}||_{2}
\\\geq&
\int_{\{\psi\geq-t_{0}\}}|\tilde{F}_{j}|^{2}e^{-\varphi}c(-\psi)+
||\tilde{F}_{j}||_{2}^{2}-||F_{t_{0}}||_{2}^{2}
\\&-2||\tilde{F}_{j}-F_{t_{0}}||_{2}
(\int_{\{-t_{0}-B_{j}<\psi<-t_{0}\}}|F_{t_{0}}|^{2}e^{-\varphi}c(-\psi))^{1/2}
\\=&
\int_{M}|\tilde{F}_{j}|^{2}e^{-\varphi}c(-\psi)-||F_{t_{0}}||_{2}^{2}
\\&-2||\tilde{F}_{j}-F_{t_{0}}||_{2}
(\int_{\{-t_{0}-B_{j}<\psi<-t_{0}\}}|F_{t_{0}}|^{2}e^{-\varphi}c(-\psi))^{1/2}.
\end{split}
\end{equation}

It follows from equality \ref{equ:20170913e} that
\begin{equation}
\label{equ:20170917a}
\begin{split}
||\tilde{F}_{j}-F_{t_{0}}||_{2}
=(||\tilde{F}_{j}||_{2}^{2}-||F_{t_{0}}||_{2}^{2})^{1/2}
\leq||\tilde{F}_{j}||_{2}\leq(\int_{M}|\tilde{F}_{j}|^{2}e^{-\varphi}c(-\psi))^{1/2}.
\end{split}
\end{equation}
Since $\int_{M}|\tilde{F}_{j}|^{2}e^{-\varphi}c(-\psi)$ is bounded with respect to $j$,
inequality \ref{equ:20170917a} implies that
$(\int_{\{\psi<-t_{0}\}}|\tilde{F}_{j}-F_{t_{0}}|^{2}e^{-\varphi}c(-\psi))^{1/2}$ is bounded with respect to $j$.
Using the dominated convergence theorem and $\int_{\{\psi<-t_{0}\}}|F_{t_{0}}|^{2}e^{-\varphi}c(-\psi)=G(t_{0})\leq G(0)<+\infty$,
we obtain that
$\lim_{j\to +\infty}\int_{\{-t_{0}-B_{j}<\psi<-t_{0}\}}|F_{t_{0}}|^{2}e^{-\varphi}c(-\psi)=0.$
Then
$$\lim_{j\to +\infty}||\tilde{F}_{j}-F_{t_{0}}||_{2}
(\int_{\{-t_{0}-B_{j}<\psi <-t_{0}\}}|F_{t_{0}}|^{2}e^{-\varphi}c(-\psi))^{1/2}=0.$$
Combining with inequality \ref{equ:20170915e},
we obtain
\begin{equation}
\label{equ:20170916a}
\begin{split}
&\liminf_{j\to+\infty}\int_{M}|\tilde{F}_{j}-(1-b_{t_{0},B_{j}}(\psi))F_{t_{0}}|^{2}e^{-\varphi}c(-\psi)
\\\geq&\liminf_{j\to+\infty}\int_{M}|\tilde{F}_{j}|^{2}e^{-\varphi}c(-\psi)-||F_{t_{0}}||_{2}^{2}.
\end{split}
\end{equation}

Using inequality \ref{equ:GZc} (3rd $``\geq"$) and inequality \ref{equ:20170916a} (4th $``\geq"$), we obtain
\begin{equation}
\label{equ:20170912b}
\begin{split}
&\frac{\int_{T}^{t_{0}}c(t)e^{-t}dt}{c(t_{0})e^{-t_{0}}}
\lim_{j\to+\infty}(\frac{G(t_{0})-G(t_{0}+B_{j})}{B_{j}})
\\=&\lim_{j\to+\infty}\frac{e^{t_{0}+B_{j}}\int_{T}^{t_{0}+B_{j}}c(t)e^{-t}dt}{\inf_{t\in(t_{0},t_{0}+B_{j})}c(t)}
(\frac{G(t_{0})-G(t_{0}+B_{j})}{B_{j}})
\\\geq&\liminf_{j\to+\infty}\frac{e^{t_{0}+B_{j}}\int_{T}^{t_{0}+B_{j}}c(t)e^{-t}dt}{\inf_{t\in(t_{0},t_{0}+B_{j})}c(t)}
\int_{M}\frac{1}{B_{j}}(\mathbb{I}_{\{-t_{0}-B<\psi<-t_{0}\}})|F_{t_{0}}|^{2}e^{-\varphi}c(-\psi)
\\\geq&\liminf_{j\to+\infty}\int_{M}|\tilde{F}_{j}-(1-b_{t_{0},B_{j}}(\psi))F_{t_{0}}|^{2}e^{-\varphi}c(-\psi)
\\\geq&\liminf_{j\to+\infty}\int_{M}|\tilde{F}_{j}|^{2}e^{-\varphi}c(-\psi)-||F_{t_{0}}||_{2}^{2}
\\\geq&G(T)-G(t_{0}).
\end{split}
\end{equation}
This proves Lemma \ref{lem:C}.
\end{proof}

Lemma \ref{lem:C} implies the following lemma.

\begin{Lemma}
\label{lem:Da}
Assume that $G(T)<+\infty$.
Then for any $t_{0},t_{1}\in[T,+\infty)$,
we have
$$\frac{G(t_{1})-G(t_{1}+t_{0})}{\int_{t_{1}}^{t_{1}+t_{0}}c(t)e^{-t}dt}\leq \frac{\liminf_{B\to0+0}(\frac{G(t_{0}+t_{1})-G(t_{0}+t_{1}+B)}{B})}{c(t_{0})e^{-t_{0}}},$$
i.e.
\begin{equation}
\label{equ:20171012b}
\begin{split}
&\frac{G(t_{1})-G(t_{1}+t_{0})}{\int_{t_{1}}^{+\infty}c(t)e^{-t}dt-\int_{t_{1}+t_{0}}^{+\infty}c(t)e^{-t}dt}
\\\leq&\liminf_{B\to0+0}
\frac{G(t_{0}+t_{1})-G(t_{0}+t_{1}+B)}{\int_{t_{1}+t_{0}}^{+\infty}c(t)e^{-t}dt-\int_{t_{1}+t_{0}+B}^{+\infty}c(t)e^{-t}dt}
\end{split}
\end{equation}
\end{Lemma}

\subsection{Proof of Theorem \ref{thm:general_concave}}

As $G(g^{-1}(r);c)$ is lower semicontinuous (Lemma \ref{lem:B}),
then it follows from the following well-known property of concave functions (Lemma \ref{lem:Ea}) that
Lemma \ref{lem:Da} implies Theorem \ref{thm:general_concave}.

\begin{Lemma}
\label{lem:Ea}
Let $a(r)$ be a lower semicontinuous function on $(0,R]$.
Then $a(r)$ is concave if and only if
$$\frac{a(r_{1})-a(r_{2})}{r_{1}-r_{2}}\leq \liminf_{r_{3}\to r_{2}-0}\frac{a(r_{3})-a(r_{2})}{r_{3}-r_{2}},$$
holds for any $0< r_{2}<r_{1}\leq R$.
\end{Lemma}

\section{Appendix: Proof of Lemma \ref{lem:GZ_sharp}}

In this section, we prove Lemma \ref{lem:GZ_sharp}.

\subsection{Preparations}
It follows from Lemma \ref{l:FN1}
that there exist smooth plurisubharmonic functions $\psi_{m}$ and $\varphi_{m}$ on $M$ decreasing convergent to $\psi$ and $\varphi$ respectively.

The following remark shows that it suffices to consider Lemma \ref{lem:GZ_sharp} for the case that
$M$ is a relatively compact open Stein submanifold of a Stein manifold,
and $F$ is a holomorphic $(n,0)$ form on $\{\psi<-t_{0}\}$ such that $\int_{\{\psi<-t_{0}\}}|F|^{2}<+\infty$,
which implies that
$\sup_{m}\sup_{M}\psi_{m}<-T$ and $\sup_{m}\sup_{M}\varphi_{m}<+\infty$ on $M$.

In the following remark, we recall some standard steps (see e.g. \cite{siu96,guan-zhou13p,guan-zhou13ap}) to illustrate it.

\begin{Remark}
\label{rem:unify}
It is well-known that there exist open Stein submanifolds $D_{1}\subset\subset\cdots\subset\subset D_{j}\subset\subset D_{j+1}\subset\subset\cdots$
such that $\cup_{j=1}^{+\infty}D_{j}=M$.

If inequality \eqref{equ:3.4} holds on any $D_{j}$ and inequality \eqref{equ:20171122a} holds on $M$,
then for any $B>0$, we obtain a sequence of holomorphic (n,0) forms $\tilde{F}_{j}$ on $D_{j}$ such that
\begin{equation}
\begin{split}
&\int_{D_{j}}|\tilde{F}_{j}-(1-b(\psi))F|^{2}e^{-\varphi+v(\psi)}c(-v(\psi))d\lambda_{n}
\\\leq&\int_{T}^{t_{0}+B}c(t)e^{-t}dt\int_{D_{j}}\frac{1}{B}\mathbb{I}_{\{-t_{0}-B<\psi<-t_{0}\}}|F|^{2}e^{-\varphi}d\lambda_{n}\leq C\int_{T}^{t_{0}+B}c(t)e^{-t}dt
\end{split}
\end{equation}
is bounded with respect to $j$.
Note that for any given $j$, $e^{-\varphi+v(\psi)}c(-v(\psi))$ has a positive lower bound,
then it follows that for any any given $j$, $\int_{ D_{j}}|\tilde{F}_{j'}-(1-b(\psi))F|^{2}$ is bounded with respect to $j'\geq j$.
Combining with
\begin{equation}
\label{equ:20171123a}
\int_{ D_{j}}|(1-b(\psi))F|^{2}\leq
\int_{D_{j}\cap\{\psi<-t_{0}\}}|F|^{2}<+\infty
\end{equation}
and inequality \eqref{equ:3.4},
one can obtain that $\int_{ D_{j}}|\tilde{F}_{j'}|^{2}$ is bounded with respect to $j'\geq j$.

By diagonal method, there exists a subsequence $F_{j''}$ uniformly convergent on any $\bar{M}_{j}$ to a holomorphic $(n,0)$ form on $M$ denoted by $\tilde{F}$.
Then it follows from inequality \eqref{equ:20171123a} and the dominated convergence theorem that
$$\int_{ D_{j}}|\tilde{F}-(1-b(\psi))F|^{2}e^{-\max\{\varphi-v(\psi),-N\}}c(-v(\psi))\leq C\int_{T}^{t_{0}+B}c(t)e^{-t}dt$$
for any $N>0$,
which implies
$$\int_{ D_{j}}|\tilde{F}-(1-b(\psi))F|^{2}e^{-(\varphi-v(\psi))}c(-v(\psi))\leq C\int_{T}^{t_{0}+B}c(t)e^{-t}dt,$$
then one can obtain Lemma \ref{lem:GZ_sharp} when $j$ goes to $+\infty$.
\end{Remark}

For the sake of completeness, we recall some lemmas on $L^{2}$ estimates for some $\bar\partial$ equations, and $\bar\partial^*$ means the Hilbert adjoint operator of
$\bar\partial$.

In the following part of this subsection, we recall some lemmas on $L^{2}$ estimates for
some $\bar\partial$ equations. Denote by
$\bar\partial^*$ or $D''^{*}$ means the Hilbert adjoint operator of
$\bar\partial$.

\begin{Lemma}\label{l:vector}(see \cite{ohsawa3} or \cite{ohsawa-takegoshi})
Let $(X,\omega)$ be a K\"{a}hler manifold of dimension n with a
K\"{a}hler metric $\omega$. Let $(E,h)$ be a hermitian holomorphic
vector bundle. Let $\eta,g>0$ be smooth functions on $X$. Then for
every form $\alpha\in \mathcal{D}(X,\Lambda^{n,q}T_{X}^{*}\otimes
E)$, which is the space of smooth differential forms with values in
$E$ with compact support, we have

\begin{equation}
\label{}
\begin{split}
&\|(\eta+g^{-1})^{\frac{1}{2}}D''^{*}\alpha\|^{2}
+\|\eta^{\frac{1}{2}}D''\alpha\|^{2}
\\&\geq\ll[\eta\sqrt{-1}\Theta_{E}-\sqrt{-1}\partial\bar\partial\eta-
\sqrt{-1}g\partial\eta\wedge\bar\partial\eta,\Lambda_{\omega}]\alpha,\alpha\gg.
\end{split}
\end{equation}
\end{Lemma}

\begin{Lemma}
\label{l:positve}(see \cite{guan-zhou13ap})Let $X$ and $E$ be as in the above lemma and
$\theta$ be a continuous $(1,0)$ form on $X$. Then we have
$$[\sqrt{-1}\theta\wedge\bar\theta,\Lambda_{\omega}]\alpha=\bar\theta\wedge(\alpha\llcorner(\bar\theta)^\sharp\big),$$
for any $(n,1)$ form $\alpha$ with value in $E$. Moveover, for any
positive $(1,1)$ form $\beta$, we have $[\beta,\Lambda_{\omega}]$ is
semipositive.
\end{Lemma}

\begin{Lemma}\label{l:vector7}(see \cite{demailly99,demailly2010})
Let $X$ be a complete K\"{a}hler manifold equipped with a (non
necessarily complete) K\"{a}hler metric $\omega$, and let $E$ be a
Hermitian vector bundle over $X$. Assume that there are smooth and
bounded functions $\eta$, $g>0$ on $X$ such that the (Hermitian)
curvature operator

$$\textbf{B}:=[\eta\sqrt{-1}\Theta_{E}-\sqrt{-1}\partial\bar\partial\eta-
\sqrt{-1}g\partial\eta\wedge\bar\partial\eta,\Lambda_{\omega}]$$
is positive definite everywhere on $\Lambda^{n,q}T^{*}_{X}\otimes E$, for some $q\geq 1$.
Then for every form $\lambda\in L^{2}(X,\Lambda^{n,q}T^{*}_{X}\otimes E)$ such that $D''\lambda=0$ and
$\int_{X}\langle\textbf{B}^{-1}\lambda,\lambda\rangle dV_{\omega}<\infty$,
there exists $u\in L^{2}(X,\Lambda^{n,q-1}T^{*}_{X}\otimes E)$ such that $D''u=\lambda$ and
$$\int_{X}(\eta+g^{-1})^{-1}|u|^{2}dV_{\omega}\leq\int_{X}\langle\textbf{B}^{-1}\lambda,\lambda\rangle dV_{\omega}.$$
\end{Lemma}\emph{}

The following Lemma belongs to Fornaess
and Narasimhan on approximation property of plurisubharmonic
functions of Stein manifolds.

\begin{Lemma}
\label{l:FN1}\cite{FN1980} Let $X$ be a Stein manifold and $\varphi \in PSH(X)$. Then there exists a sequence
$\{\varphi_{n}\}_{n=1,2,\cdots}$ of smooth strongly plurisubharmonic functions such that
$\varphi_{n} \downarrow \varphi$.
\end{Lemma}

\subsection{Proof of Lemma \ref{lem:GZ_sharp}}

For the sake of completeness, let's recall some steps in the proof in
\cite{guan_sharp} (see also \cite{guan-zhou13p,guan-zhou13ap,GZopen-effect}) with some slight
modifications in order to prove Lemma \ref{lem:GZ_sharp}.

It follows from Remark \ref{rem:unify} that
it suffices to consider that $M$ is a Stein manifold, and $F$ is holomorphic $(n,0)$ form on $U\cap \{\psi<-t_{0}\}$ and
\begin{equation}
\label{equ:20171122e}
\int_{\{\psi<-t_{0}\}}|F|^{2}<+\infty,
\end{equation}
and there exist smooth plurisubharmonic functions $\psi_{m}$ and $\varphi_{m}$ on $M$ decreasing convergent to $\psi$ and $\varphi$ respectively,
satisfying $\sup_{m}\sup_{M}\psi_{m}<-T$ and $\sup_{m}\sup_{M}\varphi_{m}<+\infty$.

\

\emph{Step 1: recall some Notations}

\

Let $\varepsilon\in(0,\frac{1}{8}B)$.
Let $\{v_{\varepsilon}\}_{\varepsilon\in(0,\frac{1}{8}B)}$ be a family of smooth increasing convex functions on $\mathbb{R}$,
which are continuous functions on $\mathbb{R}\cup\{-T\}$, such that:

 $1).$ $v_{\varepsilon}(t)=t$ for $t\geq-t_{0}-\varepsilon$, $v_{\varepsilon}(t)=constant$ for $t<-t_{0}-B+\varepsilon$;

 $2).$ $v''_{\varepsilon}(t)$ are pointwise convergent to $\frac{1}{B}\mathbb{I}_{(-t_{0}-B,-t_{0})}$, when $\varepsilon\to 0$,
 and $0\leq v''_{\varepsilon}(t)\leq \frac{2}{B}\mathbb{I}_{(-t_{0}-B+\varepsilon,-t_{0}-\varepsilon)}$ for any $t\in \mathbb{R}$;

 $3).$ $v'_{\varepsilon}(t)$ are pointwise convergent to $b(t)$ which is a continuous function on $\mathbb{R}$, when $\varepsilon\to 0$, and $0\leq v'_{\varepsilon}(t)\leq1$ for any $t\in \mathbb{R}$.

One can construct the family $\{v_{\varepsilon}\}_{\varepsilon\in(0,\frac{1}{8}B)}$ by the setting
\begin{equation}
\label{equ:20140101}
\begin{split}
v_{\varepsilon}(t):=&\int_{-\infty}^{t}(\int_{-\infty}^{t_{1}}(\frac{1}{B-4\varepsilon}
\mathbb{I}_{(-t_{0}-B+2\varepsilon,-t_{0}-2\varepsilon)}*\rho_{\frac{1}{4}\varepsilon})(s)ds)dt_{1}
\\&-\int_{-\infty}^{0}(\int_{-\infty}^{t_{1}}(\frac{1}{B-4\varepsilon}\mathbb{I}_{(-t_{0}-B+2\varepsilon,
-t_{0}-2\varepsilon)}*\rho_{\frac{1}{4}\varepsilon})(s)ds)dt_{1},
\end{split}
\end{equation}
where $\rho_{\frac{1}{4}\varepsilon}$ is the kernel of convolution satisfying $supp(\rho_{\frac{1}{4}\varepsilon})\subset (-\frac{1}{4}\varepsilon,\frac{1}{4}\varepsilon)$.
Then it follows that
$$v''_{\varepsilon}(t)=\frac{1}{B-4\varepsilon}\mathbb{I}_{(-t_{0}-B+2\varepsilon,-t_{0}-2\varepsilon)}*\rho_{\frac{1}{4}\varepsilon}(t),$$
and
$$v'_{\varepsilon}(t)=\int_{-\infty}^{t}(\frac{1}{B-4\varepsilon}\mathbb{I}_{(-t_{0}-B+2\varepsilon,-t_{0}-2\varepsilon)}
*\rho_{\frac{1}{4}\varepsilon})(s)ds.$$

It suffices to consider the case that
\begin{equation}
\label{equ:20171121a}
\int_{M}\frac{1}{B}\mathbb{I}_{\{-t_{0}-B<\psi<-t_{0}\}}|F|^{2}e^{-\psi-\varphi}<+\infty.
\end{equation}

Let $\eta=s(-v_{\varepsilon}(\psi_{m}))$ and $\phi=u(-v_{\varepsilon}(\psi_{m}))$,
where $s\in C^{\infty}((T,+\infty))$ satisfies $s\geq0$, and
$u\in C^{\infty}((T,+\infty))$, satisfies $\lim_{t\to+\infty}u(t)$ exists, such that $u''s-s''>0$, and $s'-u's=1$.
It follows from $\sup_{m}\sum_{M}\psi_{m}<-T$ that $\phi=u(-v_{\varepsilon}(\psi_{m}))$ are uniformly bounded
on $M$ with respect to $m$ and $\varepsilon$,
and $u(-v_{\varepsilon}(\psi))$ are uniformly bounded
on $M$ with respect to $\varepsilon$.
Let $\Phi=\phi+\varphi_{m'}$,
and let $\tilde{h}=e^{-\Phi}$.

\

\emph{Step 2: Solving $\bar\partial-$equation with smooth polar function and smooth weight}

\
Now let $\alpha\in \mathcal{D}(M,\Lambda^{n,1}T_{M}^{*})$ be a
smooth $(n,1)$ form with compact support on $M$. Using Lemma
\ref{l:vector} and Lemma \ref{l:positve}, the inequality $s\geq0$ and the fact that $\varphi_{m}$ is
plurisubharmonic on $M$, we get

\begin{equation}
\label{equ:10.1}
\begin{split}
&\|(\eta+g^{-1})^{\frac{1}{2}}D''^{*}\alpha\|^{2}_{M,\tilde{h}}
+\|\eta^{\frac{1}{2}}D''\alpha\|^{2}_{M,\tilde{h}}
\\&\geq\ll[\eta\sqrt{-1}\Theta_{\tilde{h}}-\sqrt{-1}\partial\bar\partial\eta-
\sqrt{-1}g\partial\eta\wedge\bar\partial\eta,\Lambda_{\omega}]\alpha,\alpha\gg_{M,\tilde{h}}
\\&\geq\ll[\eta\sqrt{-1}\partial\bar\partial\phi-\sqrt{-1}\partial\bar\partial\eta-
\sqrt{-1}g\partial\eta\wedge\bar\partial\eta,\Lambda_{\omega}]\alpha,\alpha\gg_{M,\tilde{h}}.
\end{split}
\end{equation}
where $g$ is a positive continuous function on $M$.
We need the following calculations to determine $g$.
\begin{equation}
\label{}
\begin{split}
&\partial\bar{\partial}\eta=-s'(-v_{\varepsilon}(\psi_{m}))\partial\bar{\partial}(v_{\varepsilon}(\psi_{m}))
+s''(-v_{\varepsilon}(\psi_{m}))\partial v_{\varepsilon}(\psi_{m})\wedge
\bar{\partial}v_{\varepsilon}(\psi_{m}),
\end{split}
\end{equation}

and
\begin{equation}
\label{}
\begin{split}
&\partial\bar{\partial}\phi=-u'(-v_{\varepsilon}(\psi_{m}))\partial\bar{\partial}v_{\varepsilon}(\psi_{m})
+
u''(-v_{\varepsilon}(\psi_{m}))\partial v_{\varepsilon}(\psi_{m})\wedge\bar{\partial}v_{\varepsilon}(\psi_{m}).
\end{split}
\end{equation}
Then we have
\begin{equation}
\label{equ:vector1}
\begin{split}
&-\partial\bar{\partial}\eta+\eta\partial\bar{\partial}\phi-g(\partial\eta)\wedge\bar\partial\eta
\\=&(s'-su')\partial\bar{\partial}v_{\varepsilon}(\psi_{m})+((u''s-s'')-gs'^{2})\partial
(-v_{\varepsilon}(\psi_{m}))\bar{\partial}(-v_{\varepsilon}(\psi_{m}))
\\=&(s'-su')(v'_{\varepsilon}(\psi_{m})\partial\bar{\partial}\psi_{m}+v''_{\varepsilon}(\psi_{m})
\partial(\psi_{m})\wedge\bar{\partial}(\psi_{m}))
\\&+((u''s-s'')-gs'^{2})\partial
(-v_{\varepsilon}(\psi_{m}))\wedge\bar{\partial}(-v_{\varepsilon}(\psi_{m})).
\end{split}
\end{equation}
We omit composite item $-v_{\varepsilon}(\psi_{m})$ after $s'-su'$ and $(u''s-s'')-gs'^{2}$ in the above equalities.

As $v'_{t_0,\varepsilon}\geq 0$  and $s'-su'=1$, using Lemma
\ref{l:positve}, equality \eqref{equ:vector1} and inequality
\eqref{equ:10.1}, we obtain
\begin{equation}
\label{equ:semi.vector3}
\begin{split}
\langle\textbf{B}\alpha, \alpha\rangle_{\tilde{h}}=&\langle[\eta\sqrt{-1}\Theta_{\tilde{h}}-\sqrt{-1}\partial\bar\partial
\eta-\sqrt{-1}g\partial\eta\wedge\bar\partial\eta,\Lambda_{\omega}]
\alpha,\alpha\rangle_{\tilde{h}}
\\\geq&
\langle[(v''_{t_0,\varepsilon}\circ\psi_{m})
\sqrt{-1}\partial\psi_{m}\wedge\bar{\partial}\psi_{m},\Lambda_{\omega}]\alpha,\alpha\rangle_{\tilde{h}}
\\=&\langle (v''_{t_{0},\varepsilon}\circ\psi_{m}) \bar\partial\psi_{m}\wedge
(\alpha\llcorner(\bar\partial\psi_{m})^\sharp\big ),\alpha\rangle_{\tilde{h}}.
\end{split}
\end{equation}

Using the definition of contraction, Cauchy-Schwarz inequality and
the inequality \eqref{equ:semi.vector3}, we have
\begin{equation}
\label{equ:20181106}
\begin{split}
|\langle (v''_{t_{0},\varepsilon}\circ\psi_{m})\bar\partial\psi_{m}\wedge \gamma,\tilde{\alpha}\rangle_{\tilde{h}}|^{2}
=&|\langle (v''_{t_{0},\varepsilon}\circ\psi_{m}) \gamma,\tilde{\alpha}\llcorner(\bar\partial\psi_{m})^\sharp\big
\rangle_{\tilde{h}}|^{2}
\\\leq&\langle( v''_{t_{0},\varepsilon}\circ\psi_{m}) \gamma,\gamma\rangle_{\tilde{h}}
(v''_{t_{0},\varepsilon}\circ\psi_{m})|\tilde{\alpha}\llcorner(\bar\partial\psi_{m})^\sharp\big|_{\tilde{h}}^{2}
\\=&\langle (v''_{t_{0},\varepsilon}\circ\psi_{m}) \gamma,\gamma\rangle_{\tilde{h}}
\langle (v''_{t_{0},\varepsilon}\circ\psi_{m}) \bar\partial\psi_{m}\wedge
(\tilde{\alpha}\llcorner(\bar\partial\psi_{m})^\sharp\big ),\tilde{\alpha}\rangle_{\tilde{h}}
\\\leq&\langle (v''_{t_{0},\varepsilon}\circ\psi_{m})\gamma,\gamma\rangle_{\tilde{h}}
\langle\textbf{B}\tilde{\alpha},\tilde{\alpha}\rangle_{\tilde{h}},
\end{split}
\end{equation}
for any $(n,0)$ form $\gamma$ and $(n,1)$ form $\tilde{\alpha}$.

As $F$ is holomorphic on $\{\psi<-t_{0}\}\supset\supset Supp(v'_{\varepsilon}(\psi_{m}))$,
then $\lambda:=\bar{\partial}[(1-v'_{\varepsilon}(\psi_{m})){F}]$
is well-defined and smooth on $M$.

Taking $\gamma=F$, and
$\tilde{\alpha}=\textbf{B}^{-1}\bar\partial\Psi\wedge \tilde{F}$,
note that $\tilde{h}=e^{-\Phi}$,
using inequality \eqref{equ:20181106},
we have
$$\langle \textbf{B}^{-1}\lambda,\lambda\rangle_{\tilde{h}} \leq v''_{t_0,\varepsilon}(\psi_{m})| \tilde{F}|^{2}e^{-\Phi}.$$
Then it follows that
$$\int_{M}\langle \textbf{B}^{-1}\lambda,\lambda\rangle_{\tilde{h}}
 \leq \int_{M}v''_{t_0,\varepsilon}(\psi_{m})| \tilde{F}|^{2}e^{-\Phi}.$$
Using Lemma \ref{l:vector7},
we have locally $L^{1}$ function $u_{m,m',\varepsilon}$ on $M$ such that $\bar{\partial}u_{m,m',\varepsilon}=\lambda$,
and
\begin{equation}
 \label{equ:3.2}
 \begin{split}
 &\int_{M}|u_{m,m',\varepsilon}|^{2}(\eta+g^{-1})^{-1} e^{-\Phi}\leq\int_{M}\langle \textbf{B}^{-1}\lambda,\lambda\rangle_{\tilde{h}}
  \leq\int_{M}v''_{\varepsilon}(\psi_{m})| F|^2e^{-\Phi}.
  \end{split}
\end{equation}

Let $g=\frac{u''s-s''}{s'^{2}}(-v_{\varepsilon}(\psi_{m}))$.
It follows that $\eta+g^{-1}=(s+\frac{s'^{2}}{u''s-s''})(-v_{\varepsilon}(\psi_{m}))$.
Let $\mu:=(\eta+g^{-1})^{-1}$.

Assume that we can choose $\eta$ and $\phi$ such that $e^{v_{\varepsilon}\circ\psi_{m}}e^{\phi}c(-v_{\varepsilon}\circ\psi_{m})=(\eta+g^{-1})^{-1}$.
Then inequality \eqref{equ:3.2} becomes
\begin{equation}
 \label{equ:20171122b}
 \begin{split}
 &\int_{M}|u_{m,m',\varepsilon}|^{2}e^{v_{\varepsilon}(\psi_{m})-\varphi_{m'}}c(-v_{\varepsilon}\circ\psi_{m})
  \leq\int_{M}v''_{\varepsilon}(\psi_{m})| F|^2e^{-\phi-\varphi_{m'}}.
  \end{split}
\end{equation}

Let $F_{m,m',\varepsilon}:=-u_{m,m',\varepsilon}+(1-v'_{\varepsilon}(\psi_{m})){F}$.
Then inequality \eqref{equ:20171122b} becomes
\begin{equation}
 \label{equ:20171122c}
 \begin{split}
 &\int_{M}|F_{m,m',\varepsilon}-(1-v'_{\varepsilon}(\psi_{m})){F}|^{2}e^{v_{\varepsilon}(\psi_{m})-\varphi_{m'}}c(-v_{\varepsilon}\circ\psi_{m})
  \\&\leq\int_{M}(v''_{\varepsilon}(\psi_{m}))| F|^2e^{-\phi-\varphi_{m'}}.
  \end{split}
\end{equation}

\

\emph{Step 3: Singular polar function and smooth weight}

\

As $\sup_{m,\varepsilon}|\phi|=\sup_{m,\varepsilon}|u(-v_{\varepsilon}(\psi_{m}))|<+\infty$ and $\varphi_{m'}$ is continuous on $\bar{M}$,
then $\sup_{m,\varepsilon}e^{-\phi-\varphi_{m'}}<+\infty$.
Note that
$$v''_{\varepsilon}(\psi_{m})| F|^2e^{-\phi-\varphi_{m'}}\leq\frac{2}{B}\mathbb{I}_{\{\psi<-t_{0}\}}| F|^{2}\sup_{m,\varepsilon}e^{-\phi-\varphi_{m'}}$$
on $M$,
then it follows from inequality \eqref{equ:20171122e} and the dominated convergence theorem that
\begin{equation}
\label{equ:20171122f}
 \lim_{m\to+\infty}\int_{M}v''_{\varepsilon}(\psi_{m})| F|^2e^{-\phi-\varphi_{m'}}=
 \int_{M}v''_{\varepsilon}(\psi)| F|^2e^{-u(-v_{\varepsilon}(\psi))-\varphi_{m'}}
\end{equation}

Note that $\inf_{m}\inf_{M}e^{v_{\varepsilon}(\psi_{m})-\varphi_{m'}}c(-v_{\varepsilon}\circ\psi_{m})>0$,
then it follows from inequality \eqref{equ:20171122c} and \eqref{equ:20171122f}
that $\sup_{m}\int_{M}|F_{m,m',\varepsilon}-(1-v'_{\varepsilon}(\psi_{m})){F}|^{2}<+\infty$.
Note that
\begin{equation}
\label{equ:20171122g}
|(1-v'_{\varepsilon}(\psi_{m}))F|\leq |\mathbb{I}_{\{\psi<-t_{0}\}}F|,
\end{equation}
then it follows from inequality \eqref{equ:20171122e}
that $\sup_{m}\int_{M}|F_{m,m',\varepsilon}|^{2}<+\infty$,
which implies that there exists a subsequence of $\{F_{m,m',\varepsilon}\}_{m}$
(also denoted by $F_{m,m',\varepsilon}$) compactly convergent to a holomorphic $F_{m',\varepsilon}$ on $M$.

Note that $v_{\varepsilon}(\psi_{m})-\varphi_{m'}$ are uniformly bounded on $M$ with respect to $m$,
then it follows from
$|F_{m,m',\varepsilon}-(1-v'_{\varepsilon}(\psi_{m})){F}|^{2}\leq
 2(|F_{m,m',\varepsilon}|^{2}+ |(1-v'_{\varepsilon}(\psi_{m})){F}|^{2})
\leq  2(|F_{m,m',\varepsilon}|^{2}+  |\mathbb{I}_{\{\psi<-t_{0}\}}F^{2}|)$
and the dominated convergence theorem that
\begin{equation}
 \label{equ:20171122d}
 \begin{split}
 \lim_{m\to+\infty}&\int_{K}|F_{m,m',\varepsilon}-(1-v'_{\varepsilon}(\psi_{m})){F}|^{2}e^{v_{\varepsilon}(\psi_{m})-\varphi_{m'}}c(-v_{\varepsilon}\circ\psi_{m})
  \\=&\int_{K}|F_{m',\varepsilon}-(1-v'_{\varepsilon}(\psi)){F}|^{2}e^{v_{\varepsilon}(\psi)-\varphi_{m'}}c(-v_{\varepsilon}\circ\psi)
  \end{split}
\end{equation}
holds for any compact subset $K$ on $M$.
Combining with inequality \eqref{equ:20171122c} and \eqref{equ:20171122f},
one can obtain that
\begin{equation}
 \label{equ:20171122h}
 \begin{split}
&\int_{K}|F_{m',\varepsilon}-(1-v'_{\varepsilon}(\psi)){F}|^{2}e^{v_{\varepsilon}(\psi)-\varphi_{m'}}c(-v_{\varepsilon}\circ\psi)
\\&\leq
\int_{M}v''_{\varepsilon}(\psi)| F|^2e^{-u(-v_{\varepsilon}(\psi))-\varphi_{m'}},
\end{split}
\end{equation}
which implies
\begin{equation}
 \label{equ:20171122i}
 \begin{split}
&\int_{M}|F_{m',\varepsilon}-(1-v'_{\varepsilon}(\psi)){F}|^{2}e^{v_{\varepsilon}(\psi)-\varphi_{m'}}c(-v_{\varepsilon}\circ\psi)
\\&\leq
\int_{M}v''_{\varepsilon}(\psi)| F|^2e^{-u(-v_{\varepsilon}(\psi))-\varphi_{m'}},
\end{split}
\end{equation}

\

\emph{Step 4: Nonsmooth cut-off function}

\

Note that
$\sup_{\varepsilon}\sup_{M}e^{-u(-v_{\varepsilon}(\psi))-\varphi_{m'}}<+\infty,$
and
$$v''_{\varepsilon}(\psi)| F|^2e^{-u(-v_{\varepsilon}(\psi))-\varphi_{m'}}\leq
\frac{2}{B}\mathbb{I}_{\{-t_{0}-B<\psi<-t_{0}\}}| F|^2\sup_{\varepsilon}\sup_{M}e^{-u(-v_{\varepsilon}(\psi))-\varphi_{m'}},$$
then it follows from inequality \eqref{equ:20171122e} and the dominated convergence theorem that
\begin{equation}
\label{equ:20171122j}
\begin{split}
&\lim_{\varepsilon\to0}\int_{M}v''_{\varepsilon}(\psi)| F|^2e^{-u(-v_{\varepsilon}(\psi))-\varphi_{m'}}
\\=&\int_{M}\frac{1}{B}\mathbb{I}_{\{-t_{0}-B<\psi<-t_{0}\}}|F|^2e^{-u(-v(\psi))-\varphi_{m'}}
\\\leq&(\sup_{M}e^{-u(-v(\psi))})\int_{M}\frac{1}{B}\mathbb{I}_{\{-t_{0}-B<\psi<-t_{0}\}}|F|^2e^{-\varphi_{m'}}<+\infty.
\end{split}
\end{equation}

Note that $\inf_{\varepsilon}\inf_{M}e^{v_{\varepsilon}(\psi)-\varphi_{m'}}c(-v_{\varepsilon}\circ\psi)>0$,
then it follows from inequality \eqref{equ:20171122i} and \eqref{equ:20171122j} that
$\sup_{\varepsilon}\int_{M}|F_{m',\varepsilon}-(1-v'_{\varepsilon}(\psi)){F}|^{2}<+\infty.$
Combining with
\begin{equation}
\label{equ:20171122k}
\sup_{\varepsilon}\int_{M}|(1-v'_{\varepsilon}(\psi)){F}|^{2}\leq\int_{M}\mathbb{I}_{\{\psi<-t_{0}\}}|F^{2}|<+\infty,
\end{equation}
one can obtain that $\sup_{\varepsilon}\int_{M}|F_{m',\varepsilon}|^{2}<+\infty$,
which implies that
there exists a subsequence of $\{F_{m',\varepsilon}\}_{\varepsilon\to0}$ (also denoted by $\{F_{m',\varepsilon}\}_{\varepsilon\to0}$)
compactly convergent to a holomorphic (n,0) form on $M$ denoted by $F_{m'}$.

Note that $\sup_{\varepsilon}\sup_{M}e^{v_{\varepsilon}(\psi)-\varphi_{m'}}c(-v_{\varepsilon}\circ\psi)<+\infty$ and
$|F_{m',\varepsilon}-(1-v'_{\varepsilon}(\psi)){F}|^{2}\leq 2(|F_{m',\varepsilon}|^{2}+|\mathbb{I}_{\{\psi<-t_{0}\}}F|^{2})$,
then it follows from inequality \eqref{equ:20171122k} and the dominated convergence theorem on any given $K\subset\subset D$
$($with dominant function
$$2(\sup_{\varepsilon}\sup_{K}(|F_{m',\varepsilon}|^{2})+\mathbb{I}_{\{\psi<-t_{0}\}}|F|^{2})\sup_{\varepsilon}
\sup_{M}e^{v_{\varepsilon}(\psi)-\varphi_{m'}}c(-v_{\varepsilon}\circ\psi))$$
that
\begin{equation}
\label{equ:20171122l}
\begin{split}
&\lim_{\varepsilon\to0}\int_{K}|F_{m',\varepsilon}-(1-v'_{\varepsilon}(\psi)){F}|^{2}e^{v_{\varepsilon}(\psi)-\varphi_{m'}}c(-v_{\varepsilon}\circ\psi)
\\=&\int_{K}|F_{m'}-(1-b(\psi)){F}|^{2}e^{v(\psi)-\varphi_{m'}}c(-v\circ\psi).
\end{split}
\end{equation}
Combining with inequality \eqref{equ:20171122j} and \eqref{equ:20171122i}, one can obtain that
\begin{equation}
\label{equ:20171122m}
\begin{split}
&\int_{K}|F_{m'}-(1-b(\psi)){F}|^{2}e^{v(\psi)-\varphi_{m'}}c(-v\circ\psi)
\\\leq&(\sup_{M}e^{-u(-v(\psi))})\int_{M}\frac{1}{B}\mathbb{I}_{\{-t_{0}-B<\psi<-t_{0}\}}|F|^2e^{-\varphi_{m'}}
\end{split}
\end{equation}
which implies
\begin{equation}
\label{equ:20171122n}
\begin{split}
&\int_{M}|F_{m'}-(1-b(\psi)){F}|^{2}e^{v(\psi)-\varphi_{m'}}c(-v\circ\psi)
\\\leq&(\sup_{M}e^{-u(-v(\psi))})\int_{M}\frac{1}{B}\mathbb{I}_{\{-t_{0}-B<\psi<-t_{0}\}}|F|^2e^{-\varphi_{m'}}.
\end{split}
\end{equation}

\

\emph{Step 5: Singular weight}

\

Note that
\begin{equation}
\label{equ:20171122o}
\int_{M}\frac{1}{B}\mathbb{I}_{\{-t_{0}-B<\psi<-t_{0}\}}|F|^2e^{-\varphi_{m'}}\leq\int_{M}\frac{1}{B}\mathbb{I}_{\{-t_{0}-B<\psi<-t_{0}\}}|F|^{2}e^{-\varphi}<+\infty,
\end{equation}
and $\sup_{M}e^{-u(-v(\psi))}<+\infty$,
then it from \eqref{equ:20171122n} that
$$\sup_{m'}\int_{M}|F_{m'}-(1-b(\psi)){F}|^{2}e^{v(\psi)-\varphi_{m'}}c(-v\circ\psi)<+\infty.$$
Combining with $\inf_{m'}\inf_{M}e^{v(\psi)-\varphi_{m'}}c(-v(\psi))>0$,
one can obtain that
$$\sup_{m'}\int_{M}|F_{m'}-(1-b(\psi)){F}|^{2}<+\infty.$$
Note that
\begin{equation}
\label{equ:20171122p}
\int_{M}|(1-b(\psi)){F}|^{2}\leq\int_{M}|\mathbb{I}_{\{\psi<-t_{0}\}}F|^{2} <+\infty.
\end{equation}
Then $\sup_{m'}\int_{M}|F_{m'}|^{2}<+\infty$,
which implies that there exists a compactly convergent subsequence of $\{F_{m'}\}$ denoted by $\{F_{m''}\}$,
which is convergent a holomorphic (n,0) form $\tilde{F}$ on $M$.

Note that $\sup_{m'}\sup_{M}e^{v(\psi)-\varphi_{m'}}c(-v\circ\psi)<+\infty$,
then it follows from inequality \eqref{equ:20171122p} and the
dominated convergence theorem on any given compact subset $K$ of $M$ $($with dominant function $2[\sup_{m''}\sup_{K}(|F_{m''}|^{2})+\mathbb{I}_{\{\psi<-t_{0}\}}|F|^{2}]\sup_{M}e^{v(\psi)-\varphi_{m'}}$ $)$ that
\begin{equation}
\label{equ:20171122q}
\begin{split}
&\lim_{m''\to+\infty}\int_{K}|F_{m''}-(1-b(\psi)){F}|^{2}e^{v(\psi)-\varphi_{m'}}c(-v\circ\psi)
\\=&\int_{K}|\tilde{F}-(1-b(\psi)){F}|^{2}e^{v(\psi)-\varphi_{m'}}c(-v\circ\psi).
\end{split}
\end{equation}
Note that for any $m''\geq m'$, $\varphi_{m'}\leq\varphi_{m''}$ holds,
then it follows from inequality \eqref{equ:20171122n} and \eqref{equ:20171122o}
that
\begin{equation}
\label{equ:20171122r}
\begin{split}
&\lim_{m''\to+\infty}\int_{K}|F_{m''}-(1-b(\psi)){F}|^{2}e^{v(\psi)-\varphi_{m'}}c(-v\circ\psi)
\\\leq&
\limsup_{m''\to+\infty}\int_{K}|F_{m''}-(1-b(\psi)){F}|^{2}e^{v(\psi)-\varphi_{m''}}c(-v\circ\psi)
\\\leq&
\limsup_{m''\to+\infty}(\sup_{M}e^{-u(-v(\psi))})\int_{M}\frac{1}{B}\mathbb{I}_{\{-t_{0}-B<\psi<-t_{0}\}}|F|^2e^{-\varphi_{m''}}
\\\leq&
(\sup_{M}e^{-u(-v(\psi))})C<+\infty.
\end{split}
\end{equation}
Combining with equality \eqref{equ:20171122q},
one can obtain that
$$\int_{K}|\tilde{F}-(1-b(\psi)){F}|^{2}e^{v(\psi)-\varphi_{m'}}c(-v\circ\psi)\leq(\sup_{M}e^{-u(-v(\psi))})C,$$
for any compact subset of $M$,
which implies
$$\int_{M}|\tilde{F}-(1-b(\psi)){F}|^{2}e^{v(\psi)-\varphi_{m'}}c(-v\circ\psi)\leq(\sup_{M}e^{-u(-v(\psi))})C.$$
When $m'\to+\infty$,
it follows from Levi's Theorem that
\begin{equation}
\label{equ:20171122s}
\begin{split}
\int_{M}|\tilde{F}-(1-b(\psi)){F}|^{2}e^{v(\psi)-\varphi}c(-v(\psi))\leq(\sup_{M}e^{-u(-v(\psi))})C.
\end{split}
\end{equation}

\

\emph{Step 6: ODE system}

\

It suffices to find $\eta$ and $\phi$ such that
$(\eta+g^{-1})=e^{-\psi_{m}}e^{-\phi}\frac{1}{c(-v_{\varepsilon}(\psi_{m}))}$ on $M$.
As $\eta=s(-v_{\varepsilon}(\psi_{m}))$ and $\phi=u(-v_{\varepsilon}(\psi_{m}))$,
we have $(\eta+g^{-1}) e^{v_{\varepsilon}(\psi_{m})}e^{\phi}=(s+\frac{s'^{2}}{u''s-s''})e^{-t}e^{u}\circ(-v_{\varepsilon}(\psi_{m}))$.

Summarizing the above discussion about $s$ and $u$, we are naturally led to a
system of ODEs (see \cite{guan-zhou12,guan-zhou13p,guan-zhou13ap,GZopen-effect}):
\begin{equation}
\label{GZ}
\begin{split}
&1).\,\,(s+\frac{s'^{2}}{u''s-s''})e^{u-t}=\frac{1}{c(t)}, \\
&2).\,\,s'-su'=1,
\end{split}
\end{equation}
where $t\in(T,+\infty)$.

It is not hard to solve the ODE system \ref{GZ} and get $u(t)=-\log(\int_{T}^{t}c(t_{1})e^{-t_{1}}dt_{1})$ and
$s(t)=\frac{\int_{T}^{t}(\int_{T}^{t_{2}}c(t_{1})e^{-t_{1}}dt_{1})dt_{2}}{\int_{T}^{t}c(t_{1})e^{-t_{1}}dt_{1}}$
(see \cite{guan-zhou13ap}).
It follows that $s\in C^{\infty}((T,+\infty))$ satisfies $s\geq0$, $\lim_{t\to+\infty}u(t)=-\log(\int_{T}^{+\infty}c(t_{1})e^{-t_{1}}dt_{1})$ and
$u\in C^{\infty}((T,+\infty))$ satisfies $u''s-s''>0$.

As $u(t)=-\log(\int_{T}^{t}c(t_{1})e^{-t_{1}}dt_{1})$ is decreasing with respect to $t$,
then it follows from $-T\geq v(t)\geq\max\{t,-t_{0}-B_{0}\}\geq -t_{0}-B_{0}$ for any $t\leq0$
that
\begin{equation}
\begin{split}
\sup_{M}e^{-u(-v(\psi))}
\leq\sup_{t\in(T,t_{0}+B]}e^{-u(t)}
=\int_{T}^{t_{0}+B}c(t_{1})e^{-t_{1}}dt_{1},
\end{split}
\end{equation}
therefore we are done.
Thus we prove Lemma \ref{lem:GZ_sharp}.

\vspace{.1in} {\em Acknowledgements}.
The present article was completed when the author was a member of School of Mathematics, Institute for Advanced Study.
The author would like to thank the valuable comments and sincerely help of the referee.
The author would like to thank Zhitong Mi for helpful discussions.

The author was supported by NSFC-11825101, NSFC-11522101 and NSFC-11431013,
and the National Science Foundation Grant No. DMS-163852 and the Ky Fan and Yu-Fen Fan Membership Fund.

\bibliographystyle{references}
\bibliography{xbib}

\end{document}